\documentclass[12pt,a4paper,reqno]{amsart}
\usepackage{xypic,a4wide,amsmath,amssymb,amsthm,ifthen}
\usepackage[mathcal]{euscript}
\usepackage{mathrsfs,stmaryrd}
\renewcommand{\baselinestretch}{1.3}
\newtheorem{thm}{Theorem}[section]

\newtheorem{lemma}[thm]{Lemma}
\newtheorem{prop}[thm]{Proposition}
\newtheorem{defin}[thm]{Definition}

\theoremstyle{remark}
\newtheorem{rem}[thm]{Remark}
\newtheorem{ex}[thm]{Example}
\newenvironment{remark}{\begin{rem}\rm}{\qee\end{rem}}
\newenvironment{example}{\begin{ex}\rm}{\qee\end{ex}}

\newcommand{\cO}{{\mathcal O}}

\newcommand{\CC}{{\mathcal C}^\infty}

\newcommand{\id}{\operatorname{id}}
\newcommand{\rk}{{\mbox{rk}\,}}

\newcommand{\C}{{\mathbb C}}

\newcommand{\A}{{\mathscr  A}}

\newcommand{\Lie}{{\mathcal L}}

\newcommand{\vf}{{{\mathfrak X}(M)}}
\newcommand{\qee}{\mbox{\hspace{0.2mm}}\hfill$\triangle$}

\begin{document}
\begin{flushright}SISSA Preprint 15/2010/fm\\ \tt arXiv:1003.1823 [math.CV]\end{flushright}
\title[Skew-holomorphic Lie algebroids]{\large Cohomology of skew-holomorphic\\[8pt]  Lie algebroids}
\bigskip
\maketitle \thispagestyle{empty} \vspace{-3mm}
\begin{center}{\sc Ugo Bruzzo} \\ {\small
Scuola Internazionale Superiore di Studi Avanzati,\\ Via Beirut 2-4, 34013
Trieste, Italia; \\ Istituto Nazionale di Fisica Nucleare, Sezione di Trieste \\ E-mail: {\tt bruzzo@sissa.it} }
\\[6pt]
{\sc Vladimir  Rubtsov} \\ {\small
Universit\'e d'Angers, D\'epartement de Math\'ematiques,\\
UFR Sciences, LAREMA, UMR 6093 du CNRS,\\
2 bd.~Lavoisier, 49045 Angers Cedex 01, France; \\
 ITEP Theoretical Division, \\ 25, Bol.~Tcheremushkinskaya str., 117259, Moscow, Russia
\\ E-mail: {\tt Volodya.Roubtsov@univ-angers.fr}}
\end{center}
\par
\bigskip

\vfill

\begin{abstract}  We introduce the notion of \emph{skew-holomorphic Lie algebroid}
on a complex manifold, and explore some cohomologies theories
that one can associate to it. Examples are given in terms
of holomorphic Poisson structures of various sorts. \end{abstract}

\vfill
\parbox{.75\textwidth}{\hrulefill}\par
\noindent \begin{minipage}[c]{\textwidth}\parindent=0pt \renewcommand{\baselinestretch}{1.2}
\small
{\em Date:} 6 March 2010, revised 19 July 2010  \par \small
\par\smallskip

The authors gratefully acknowledge
financial support and hospitality during the respective visits to
Universit\'e d'Angers and {\sc sissa}. Support for this work was also provided by {\sc misgam} (Methods of Integrable Systems, Geometry, Applied Mathematics), by   the {\sc infn}~project {\sc pi14} ``Nonperturbative dynamics of gauge theories", the {\sc einstein} Italo-Russian project  ``Integrability in topological string and field theory,'' and the {\sc matpyl}   Angers-{\sc sissa} project
``Lie algebroids, equivariant cohomology, and topological quantum field and string theories.''
\end{minipage}

\newpage

\section{Introduction}\label{intro}
The complex structure of a complex manifold $X$ gives rise to a rich cohomological
structure; one has the Dolbeault cohomology, the holomorphic de Rham
cohomology, and these relate in a nontrivial way to the (usual) de Rham
cohomology of $X$. A complex structure for $X$ may be regarded as an  integrable
decomposition 
\begin{equation} T_X\otimes \C = T^{1,0}_X\oplus T^{0,1}_X \label{split}\end{equation}
with the condition $\overline{T^{1,0}_X} = T^{0,1}_X$.

On the other hand one has the notion of Lie algebroid; loosely speaking (the precise
definition is recalled below),  one has a vector bundle morphism $a\colon A \to T_X$
with a lift of the Lie algebra structure on the sections of $T_X$ to a Lie algebra structure on the sections of $A$. One can then think of lifting the decomposition \eqref{split} as well.
In this paper we analyze the cohomological theory arising from such a structure. In particular,
we consider Lie algebroids that are obtained by ``matching'' --- in a specific technical sense that
we shall recall in the body of the paper --- a holomorphic Lie algebroid $\A_1$ with the complex
conjugate of another holomorphic Lie algebroid $\A_2$. We call the structure obtained
in this way a \emph{skew-holomorphic} Lie algebroid. A particular case of this
construction is presented in the paper \cite{GSX}, where $\A_2$ is assumed to be
the holomorphic tangent bundle to $X$.

The cohomology theory of skew-holomorphic algebroids turns out to be quite rich.
This paper is devoted to explore it. After recalling some basic definitions in Section \ref{preli},
in Section \ref{skew} we review the notions of \emph{representation of a Lie algebroid},
of \emph{matched pair of Lie algebroids}, and introduce the new concepts
of \emph{almost complex structure on a Lie algebroid}, and of \emph{skew-holomorphic
Lie algebroid.} In Section \ref{cohom} we give our main theorem about the cohomological
structure of such Lie algebroids, and in the final Section \ref{examples} we provide some examples,
basically related to various  holomorphic  Poisson cohomologies.

\medskip
{\bf Acknowledgments.} This paper has been mostly written while U.B~was visiting the Department of Physics of Rutgers University.  He gratefully acknowledges the Department's  hospitality
and warm welcome.  The second author thanks LPTM of the Cergy-Pontoise University for hospitality during his CNRS delegation while the paper was prepared. Part of the results was presented by V.R.~during
the Conference ``Differential Equations and Topology'' in the Institute of Mathematics, National Academy of Sciences of Ukraine. He is thankful to the CNRS PICS project ``Probl\`emes
de Physique Math\'ematique'' (France -Ukraine) for supporting his participation.

 We thank Yvette Kosmann-Schwarzbach and Mathieu Sti\'enon for useful discussions.


\bigskip

\section{Preliminaries}\label{preli}
\subsection{Lie algebroids.} We start by recalling the notions of Lie algebroid and Lie algebroid
cohomology. 
Let $M$ be a smooth manifold, $T_M$ its tangent bundle, and let $\vf$ be the space of vector fields on $M$ equipped
with the usual Lie bracket $[\,,]$.
\begin{defin} An algebroid $A$ over $M$ is a vector bundle
on $M$ together with a vector bundle morphism
$a\colon A \to T_M$ (called the anchor) and a structure of Lie algebra
on the space of global sections $\Gamma(A)$, such that
\begin{enumerate}
\item $a\colon\Gamma(A)\to \vf $ is a Lie algebra
homomorphism;
\item the following Leibniz rule holds true for every $\alpha$, $\beta\in
\Gamma(A)$ and every function $f$:
$$\{\alpha,f\beta\}=f\{\alpha,\beta\}+a(\alpha)(f)\,\beta$$
(we denote by $\{\,,\}$ the bracket in $\Gamma(A)$).
\end{enumerate}
The Lie algebroid $A$ is said to be \emph{transitive} if the anchor $a$ is
surjective.
\end{defin}
Morphisms between two Lie algebroids
$(A,a)$ and $(A',a')$ on the same base manifold $M$ are defined in a natural way, i.e.,
they are vector bundle morphisms $\phi\colon A \to A'$ such
that the map   $\phi\colon\Gamma(A)\to\Gamma(A')$ is
a Lie algebra homomorphism, and the   diagram
$$\xymatrix{A \ar[r]^\phi\ar[rd]_a & A'\ar[d]^{a'} \\ & T_M}$$ commutes.

\begin{example} An interesting example of  a transitive Lie algebroid is the \emph{Atiyah algebroid}
of a vector bundle $E$ on $M$. This is the bundle $\mathcal D(E)$
of the first-order differential operators on $E$ with scalar symbol. The   anchor
$\sigma\colon\mathcal D(E)\to T_M$ is the symbol map. Moreover,
$\ker(\sigma)\simeq End(E)$. \end{example}

 To any Lie algebroid $A$ one can associate the cohomology complex
$(C^\bullet_A,\delta)$, with $C^\bullet_A=\Gamma(\Lambda^\bullet A^\ast)$
and differential $\delta $ defined by  \cite{ELW99}
\begin{multline}\label{diff}
(\delta \xi)(\alpha_1,\dots,\alpha_{p+1}) =
\sum_{i=1}^{p+1}(-1)^{i-1}a(\alpha_i)(\xi(\alpha_1,\dots,\hat\alpha_i,
\dots,\alpha_{p+1})) \\ + \sum_{i<j}(-1)^{i+j}
\xi(\{\alpha_i,\alpha_j\},\dots,\hat\alpha_i,\dots,\hat\alpha_j,\dots,\alpha_{p+1})
\end{multline}
if $\xi\in C^p_A$ and $\alpha_i \in \Gamma(A), 1\leq i \leq p+1$. The resulting cohomology is denoted by
$H^\bullet(A)$ and is called the cohomology of the Lie algebroid $A$.

A similar definition may given in the case of a \emph{complex Lie algebroid},
where $A$ is a complex vector bundle, and $T_M$ is replaced by its
complexification $T_M\otimes\C$. Analogously, one has a notion of 
\emph{holomorphic Lie algebroid} on a complex manifold $X$, where $A$ is a holomorphic
vector bundle (that we shall denote by $\A$), $T_M$ is replaced by
the holomorphic tangent bundle $\Theta_X$, and one requires that $\A$
has a structure of sheaf of Lie algebras, satisfying a suitable Leibniz rule.

 \bigskip
\section{Skew-holomorphic Lie algebroids}\label{skew}
We shall need some results on the cohomology of holomorphic Lie algebroids.
The following theory generalizes the construction given in \cite{GSX}. Even though
we shall not need this theory in its full generality, it seems reasonable to expound it
in that form.

Let  $X$ be an $n$-dimensional compact complex manifold. We shall denote by $\Theta_X$ its holomorphic tangent bundle and by $T_X$ its tangent bundle when $X$ is regarded as a $2n$-dimensional smooth differentiable manifold. $\Omega^i_X$ will denote the bundle of holomorphic $i$-forms on $X$.  

\subsection{Almost complex Lie algebroids} There is a very natural way of extending the notion of almost complex manifold to that of almost complex Lie algebroid (this generalizes the notion of almost complex Poisson manifold given in the paper \cite{Cordero00}, to which we refer the reader for examples). Let $M$ be an almost complex manifold, with almost complex structure
$J_M\colon T_M\to T_M$.
\begin{defin} An almost complex structure $J_A$ on a real Lie algebroid
$A\stackrel{a}{\to}T_M$ is a vector bundle endomorphism $J_A\colon A\to A$ such that
$J_A^2=-\id_A$, and $J_M\circ a = a\circ J_A$.
\end{defin}
As usual we have a splitting
$$A\otimes\C = A^{1,0}\oplus A^{0,1}$$
according to the eigenvalues $\pm i$ of $J_A$. We shall set
$$\lambda_A^{p,q} = \Lambda^p(A^\ast)^{1,0}\oplus \Lambda^q(A^\ast)^{0,1}\,.$$
We set $C_A^{p,q} = \Gamma(\lambda_A^{p,q})$; 
the differential $d_A$ of the complex $C_A^\bullet = \bigoplus_{p,q}C_A^{p,q}$
splits into
$$d_A=\partial'_A+\partial_A+\bar\partial_A+\partial''_A$$
where 
$$\partial'_A\colon C_A^{p,q}\to C_A^{p+2,q-1},\qquad \partial_A\colon C_A^{p,q}\to C_A^{p+1,q}\,,$$
$$\bar\partial_A\colon C^{p,q}\to C^{p,q+1},\qquad \partial''_A\colon C^{p,q}\to C^{p-1,q+2}$$
and the identities
$$\partial_A^2=[\partial'_A,\bar\partial_A],\quad  \bar\partial_A^2=[\partial''_A,\partial_A],
\quad [\partial_A,\partial''_A]+[\partial_A,\bar\partial_A]=0$$
$$(\partial'_A)^2=(\partial''_A)^2=[\partial_A,\partial'_A]=[\bar\partial,\partial''_A]=0$$
hold. The differential complex $(C_A^\bullet,d_A)$
admits a (regular) filtration
\begin{equation}\label{filtration}F_pC_A^k = \bigoplus_{r+s=k \atop r\ge p-k} C_A^{r,s}\,.\end{equation}
A straightforward generalization of the analysis performed in \cite{LycRub94} shows the 
following.
\begin{prop} The spectral sequence associated with the filtration \eqref{filtration}
of the differential complex $(C_A^\bullet,d_A)$ converges
to the complexified cohomology of the Lie algebroid $A$, i.e., to the cohomology
$H^\bullet(A)\otimes\C$.
\end{prop}

Now let  $X$ be an $n$-dimensional   complex manifold.  

\begin{defin} \label{integ} An almost complex structure on a Lie algebroid $A$ on $X$ is
said to be integrable if  there exists a holomorphic Lie algebroid $\A$ such
that 
\begin{enumerate} \item $A^{1,0}\simeq \CC_X\otimes_{\cO_X}\A$ as sheaves of $\CC_X$-modules; 
\item under this isomorphism, the bracket of $A\otimes\C $ restricts to the bracket of $\A$;
\item the anchor $ a\colon A\otimes \C \to T^\C_X$ coincides with $\tilde a$ on $\A$.
\end{enumerate}
\end{defin}
In this case, we shall call $\A$ the \emph{holomorphic structure} of $A$. 
Note that the Lie algebroid differentials $d_A$ of $A$ and $\partial_A$ of $\A$ are related by
$$d^{1,0}_A(f\otimes\xi) = f\otimes \partial_A\xi+ a^\ast(\partial f)\wedge\xi$$
$$d^{\,0,1}_A(f\otimes\xi) = a^\ast(\bar\partial f)\wedge\xi$$
if $f$ is a smooth function, and $\xi\in \Lambda^\bullet\A^\ast$.

The integrability of an almost complex structure $J_A$ on a Lie algebroid can as usual
be detected by using a suitable Nijenhuis tensor. One defines an element
$N_A\in \Gamma(\Lambda^2(A^\ast)\otimes A)$
$$N_A(\alpha,\beta) = \{\alpha,\beta\} +J_A\{J_A\alpha,\beta\} +\{\alpha,J_A\beta\}- \{J_A\alpha,\beta\} $$
and shows that $J_A$ is integrable if and only if $N_A=0$. 

Assuming that $A$ admits a complex structure $\A$, 
 let $\Omega_\A ^k=\Lambda^k\A ^\ast$ and denote by $\partial_\A $ the differential of the Lie algebroid $\A $. So we have a complex of sheaves on $X$
\begin{equation}\label{cA}
\Omega_\A ^0 \stackrel{\partial_\A }{\to} \Omega_\A ^1  \stackrel{\partial_\A }{\to} \Omega_\A ^2 \dots
\end{equation}

Let $\partial_\A$, $d_A$ be the differentials of the Lie algebroids $\A$ and $A$, and set
set $\Omega^\bullet_\A=\Lambda^\bullet\A^\ast$. We have an injection
$\Omega^\bullet_\A\hookrightarrow \Lambda^\bullet A^\ast\otimes\C$.
\begin{lemma} If $A$ admits a holomorphic structure $\A$, then  $d_A$ restricts to $\partial_A$ on $\Omega^\bullet_\A$.
\end{lemma}
\begin{proof} 
If $f$ is a holomorphic function, and $\alpha\in\Gamma(\A)$,
$$d_A(f)(1\otimes\alpha)=a(1\otimes\alpha)(f)=\tilde a(\alpha)(f)=\partial_A(f)(\alpha)$$
so that the claim is true in degree zero. If $\xi\in\Gamma(\Omega^1_\A)$, then
\begin{eqnarray*} d_A(1\otimes\xi)(1\otimes\alpha,1\otimes\beta)  & = &  \tilde a(\alpha)(\xi(\beta))-
\tilde a(\beta)(\xi(\alpha))-\xi(\{\alpha,\beta\})   \\ & = & (1\otimes \partial_A\xi)(1\otimes\alpha,1\otimes\beta) 
\end{eqnarray*}
so that the claim is true in degree 1 as well. By the Leibniz formula one concludes.
\end{proof}

\subsection{The tangential complex of a regular holomorphic Lie algebroid} \label{tangential}
One says that a holomorphic  Lie algebroid $a\colon\A\to\Theta_X$ is
{\em regular} if the anchor $a$ has constant rank all over $X$. In this case the image
$\mathscr D$ of $\A$ in $\Theta_X$ is an involutive holomorphic subbundle of $\Theta_X$,
which is pointwise tangential to a regular holomorphic foliation $\mathcal F$ in $X$.
The differential $\partial_A$ of the sheaf complex $\Lambda^\bullet \A^\ast$ restricts
to a differential $\partial_{\mathscr D}\colon \Lambda^\bullet\mathscr D^\ast\to \Lambda^{\bullet+1}\mathscr D^\ast$. 

\begin{prop}\label{locex} \begin{enumerate} 
\item The kernel of $\partial_{\mathscr D}\colon \cO_X\to \mathscr D^\ast$
is the sheaf $\cO_{\mathcal F}$ of holomorphic functions on $X$ that are locally constant along the
leaves of $\mathcal F$;
\item the sheaf complex $( \Lambda^\bullet\mathscr D^\ast,\partial_{\mathscr D})$ is exact in positive degree;  \item there is an isomorphism $H^k(X,\cO_{\mathcal F})\simeq\mathbb H^k(X,
\Lambda^\bullet\mathscr D^\ast)$ for all $k\ge 0$.
\end{enumerate}
\end{prop}
\noindent Here    $\mathbb H^\bullet$ denotes the hypercohomology functor.
\begin{proof} Let $\Omega^\bullet_{\mathcal F}$ be the complex obtained by modding out
the holomorphic de Rham complex $(\Omega_X^\bullet,\partial)$ by the kernel of the adjoint $a^\ast$ of the anchor map. For every $k$, we call
$\Omega^k_{\mathcal F}$ the sheaf of {\em $\mathcal F$-foliated holomorphic
differential forms} on $X$. There is a ``foliated'' $\partial$-operator
$\partial_{\mathcal F}\colon \Omega^\bullet_{\mathcal F}\to  \Omega^{\bullet+1}_{\mathcal F}$, and it turns out that the adjoint of the anchor establishes an isomorphism
of complexes $( \Lambda^\bullet\mathscr D^\ast,\partial_{\mathscr D})\simeq (\Omega^\bullet_{\mathcal F},
\partial_{\mathcal F})$. We can therefore show the exactness of the complex $ (\Omega^\bullet_{\mathcal F},\partial_{\mathcal F})$
(in positive degree). 
We can also introduce the sheaves $\Omega_{\mathcal F}^{p,q}$ of {\em smooth} $\mathcal F$-foliated differential forms on $X$
that are of Hodge type $(p,q)$ (in the usual sense). We have a differential
$\bar\partial_{\mathcal F}\colon \Omega_{\mathcal F}^{\bullet,\bullet}\to \Omega_{\mathcal F}^{\bullet,\bullet+1}$.

Now, around every point of $X$ there are holomorphic coordinates $(z^1,\dots,z^m,$\break $y^1,\dots,y^{n-m})$,
where $n=\dim_\C X$, and $m=\rk_\C\mathscr D$, such that the leaves of $\mathcal F$ are given by $y^i=\mbox{const}$, and the $z$'s are coordinates on the leaves. Since the exactness of the sheaf complex
$ (\Omega^\bullet_{\mathcal F},\partial_{\mathcal F})$ is a local matter, we may assume that
$X=Z\times Y$, while  identifying the leaves of $\mathcal F$ with the complex submanifolds
$Z\times\{y\}$ for $y\in Y$. Now $(z^1,\dots,z^m)$ and $(y^1,\dots,y^{n-m})$
are local coordinates in $Z$ and $Y$, respectively. Let us note that in these coordinates
the sections of $\Omega^k_{\mathcal F}$ are written as
$$\eta = \sum \eta_{i_1,.....,i_k}(z,y)\,dz^{i_1,}\wedge\dots\wedge dz^{i_k}$$
while a section of $\Omega^{p,q}_{\mathcal F}$  is written as
$$\tau = \sum \tau_{i_1,.....,i_p,j_1,\dots,j_q}(z,\bar z,y)\,dz^{i_1,}\wedge\dots\wedge dz^{i_p}
\wedge d\bar z^{j_1}\wedge\dots\wedge d\bar z^{j_q}.$$

Claim (i) now follows. Moreover, 
it is now   easy to show that the natural map 
$\Omega^\bullet_{\mathcal F}\to \Omega^{\bullet,\bullet}_{\mathcal F}$ 
is a resolution. A standard result in homological algebra  (see, e.g., \cite[Lemma 8.5]{Voisin})
shows that the complex $\Omega^\bullet_{\mathcal F}$ is quasi-isomorphic 
to the complex of smooth complex-valued $\mathcal F$-foliated differential forms,
and as the latter is exact \cite[p.~215]{Vaisman73}, the former is exact (in positive degree) as well.

This show claim (ii).  Claim (iii) is a  straightforward consequence of the previous ones.\end{proof}

\subsection{Representations of Lie algebroids} A representation of a Lie algebroid $A\stackrel{a}{\to} T_M$ on a vector bundle
$E$ is a Lie algebroid morphism $\nabla\colon A\to\mathcal D(E)$, where 
$\mathcal D(E)\stackrel{\sigma}{\to} M $ is the Atiyah algebroid of $E$ \cite{Roub80,Mac05}. Therefore,
if $\alpha$, $s$ are sections of $A$ and $E$, respectively,
$\nabla(\alpha)$ acts on $s$; we shall denote by $\nabla_\alpha s$ this action.
The previous   abstract definition means that $\nabla$ satisfies the conditions
$$\nabla_{\{\alpha,\beta\}} = [\nabla_\alpha,\nabla_\beta],\qquad \sigma(\nabla(\alpha))=a(\alpha)\,.$$
When we have a representation of $A$ on $E$, we say as usual that $E$ is an $A$-module.

One can define a cohomology of the Lie algebroid with coefficients in $E$
by considering the twisted complex $C^\bullet_A(E)=\Gamma(\Lambda^\bullet A^\ast\otimes E^\ast)$
and defining a differential according to 
\begin{multline}\label{twdiff}
(\delta_E \xi)(\alpha_1,\dots,\alpha_{p+1},s) =
\sum_{i=1}^{p+1}(-1)^{i-1}\bigl[ a(\alpha_i)(\xi(\alpha_1,\dots,\hat\alpha_i,
\dots,\alpha_{p+1},s))  \\ - 
\xi(\alpha_1,\dots,\hat\alpha_i,\dots,\alpha_{p+1},\nabla_{\alpha_i}s)\bigr]
+ \sum_{i<j}(-1)^{i+j}
\xi(\{\alpha_i,\alpha_j\},\dots,\hat\alpha_i,\dots,\hat\alpha_j,\dots,\alpha_{p+1},s)
\end{multline}

\subsection{Matched pairs of Lie algebroids} We need the notion
of matched pair of Lie algebroids \cite{Lu97,Mac07,Mokri,Huebsch00,GSX}. We spell out
the definition in the case of real algebroids but similar constructions may be done
in the smooth complex or holomorphic cases. 
One says that $A$ and $B$ are a matched pair if $A$ is a $B$-module,
$B$ is an $A$-module, and
\newcommand{\lie}[2]{\{#1,\,#2\}}
\begin{gather}
[a(\alpha),b(\beta)] = -a\big(\nabla_\beta \alpha\big)+b\big(\nabla_\alpha \beta\big)
, \label{first} \\
\nabla_\alpha\lie{\beta_1}{\beta_2} = \lie{\nabla_\alpha \beta_1}{\beta_2} +
\lie{\beta_1}{\nabla_\alpha \beta_2} + \nabla_{\nabla_{\beta_2} \alpha}\beta_1 -
\nabla_{\nabla_{\beta_1} \alpha} \beta_2
, \label{second} \\
\nabla_\beta\lie{\alpha_1}{\alpha_2} = \lie{\nabla_\beta \alpha_1}{\alpha_2} +
\lie{\alpha_1}{\nabla_\beta \alpha_2} + \nabla_{\nabla_{\alpha_2} \beta} \alpha_1 -
\nabla_{\nabla_{\alpha_1}\beta}\alpha_2 , \label{third}
\end{gather}

If $A$, $B$ is are a matched pair of Lie algebroids, the direct sum
$A\oplus B$ can be made into a 
a Lie algebroid   $A\bowtie B$ by defining its  anchor $c$ as $c(\alpha + \beta )=a(\alpha)+b(\beta)$ and a bracket as
$$ \lie{\alpha_1+ \beta_1}{\alpha_2+ \beta_2} = \big(
\lie{\alpha_1}{\alpha_2} + \nabla_{\beta_1}\alpha_2 - \nabla_{\beta_2}\alpha_1 \big) +
\big( \lie{\beta_1}{\beta_2} + \nabla_{\alpha_1}\beta_2 - \nabla_{\alpha_2}\beta_1 \big) .
$$ 
When these conditions are satisfied, we may consider the cohomology of $A$
with coefficients in the $A$-module $\Lambda^\bullet B^\ast$, and
specularly, the cohomology of $B$
with coefficients in the $B$-module $\Lambda^\bullet A^\ast$.
The rather cumbersome conditions (\ref{first}-\ref{third}) may be
neatly stated as the condition that the  two differentials  of these complexes anticommute  \cite{GSX}. 
Thus, in the case of a matching pair of Lie algebroids, we get a 
double complex. Moreover in \cite{GSX} it is shown that the cohomology
of the total complex is isomorphic to the cohomology of the matched
sum $A\bowtie B$ (with trivial coefficients).

\subsection{Skew-holomorphic Lie algebroids}
The authors of \cite{GSX} show that any holomorphic Lie algebroid on a complex manifold $X$ can be matched
to the Lie algebroid $T^{0,1}_X$, and use that fact to develop a cohomological
theory for holomorphic Lie algebroids. This can be generalized with very little extra
cost to  study  complex Lie algebroids obtained by matching a holomorphic Lie algebroid
with an anti-holomorphic one. We call this a {\em skew-holomorphic
structure.} Of course this generalizes what happens for the complexified tangent bundle
$T_{X,\C}$. In this section we develop some elements of the cohomology of this class
of Lie algebroids.

\begin{defin} A  complex Lie algebroid $A\stackrel{a}{\to} T^\C_X$ on a complex manifold $X$ is said to have a skew-holomorphic structure if 
\begin{enumerate} \item  there is a matched pair of Lie complex algebroids $A_1$, $A_2$ such that $A\simeq A_1\bowtie A_2$;
\item $A_1 \simeq \CC\otimes \A_1$ and $\bar A_2\simeq \CC\otimes\A_2$ 
(as complex Lie algebroids) 
for some holomorphic Lie algebroids  $\A_1$, $\A_2$.
\end{enumerate}\end{defin}
Note that these conditions imply that  the anchors $a_1$, $a_2$ of $A_1$, $A_2$ satisfy $a_1(A_1)\subset T^{1,0}X$, $a_2(A_2)\subset T^{0,1}X$.

\begin{remark} \label{thetanat}  If $\A$ is a holomorphic Lie algebroid,
then $A_1=\CC_X\otimes_{\cO_X}\A$ and $T^{0,1}_X$ are matched, and therefore
one gets a complex Lie algebroid $A=A_1\bowtie T^{0,1}_X$  with a skew-holomorphic structure.
This produces the theory developed \cite{GSX}, which is thus  is a special case of ours.  
 If  in addition $\A$ is the holomorphic tangent bundle $\Theta_X$, then $A$ is the complexified smooth tangent
bundle $T^{\C}_X$, and one gets de Rham theory.
More generally, given a complex Lie algebroid $A=A_1\bowtie A_2$
with a skew-holomorphic structure, the anchor $a_2\colon A_2\to T^{0,1}X$
defines a morphism of complex Lie algebroids $A\to A_1\bowtie T^{0,1}X$. This
will in turn define a morphism $H^\bullet(A_1\bowtie T^{0,1}X)\to H^\bullet(A)$.
\end{remark}

\bigskip
\section{Local cohomology of skew-holomorphic Lie algebroids}\label{cohom}
Let $A$ be a skew-holomorphic Lie algebroid on a complex manifold $X$, and
let us consider the sheaves
$$ \lambda_A^{p,q} = \Lambda^p A^\ast_1 \otimes \Lambda^q A^\ast_2 $$
with differentials
$$\partial_A\colon  \lambda_A^{p,q}\to \lambda_A^{p+1,q},\qquad \bar\partial_A\colon  \lambda_A^{p,q}\to \lambda_A^{p,q+1}\,.$$
Since the complex Lie algebroids 
$A$, $A_2$ are matched, $(\lambda^{\bullet,\bullet},\partial_A,\bar\partial_A)$ is a double complex \cite{GSX}.
The following result is easily shown (see also  \cite{GSX}, Proposition 4.6).
\begin{prop}\label{total}  The cohomology of the Lie algebroid $A$ is isomorphic to
the cohomology of the total complex of the double complex $(\Gamma(\lambda^{\bullet,\bullet}),\partial_A,\bar\partial_A)$.
\end{prop} 

Let us denote by $\A$ the sheaf of holomorphic sections of $A_1$. Moreover, we say
that $A_2$ is \emph{transitive} if $a_2\colon A_2\to T_X^{0,1}$ is surjective. 
\begin{lemma}\label{whentrans} If $A$ has a  skew-holomorphic structure, and  $A_2$ is transitive, then $\ker [\lambda_A^{p,0 } \stackrel{\bar\partial_A}{\to} \lambda_A^{p,1 }]\simeq \Omega_\A^p$.
\end{lemma}
\begin{proof} If $f$ is a function, we have $\bar\partial_A(f)(\alpha)=a_2(\alpha)(f)$ for all $\alpha\in \Gamma (\lambda_A^{p,1 })$;
if $\bar\partial_A(f)=0$ and $a_2$ is surjective, $f$ is holomorphic. 
Let $\beta=\sum_i\alpha_i\otimes f_i$ be a section of $\lambda^{p,0}_A$. 
We may assume that the $\alpha_i$ are holomorphic (namely, they are sections of
$\Omega_\A^p$).
If $\bar\partial_A\beta=0$ then $\sum_i\alpha_i\otimes\bar\partial_Af_i=0$
which implies that the $f_i$ are holomorphic. Then $\beta=\sum_if_i\alpha_i\otimes 1$
is a section of $\Omega_A^p$.
\end{proof}

In general, without assuming that $A_2$ is transitive, let $\mathscr K^p=\ker [\lambda_A^{p,0 } \stackrel{\bar\partial_A}{\to} \lambda_A^{p,1 }]$. Since $\bar\partial_A(\Omega_\A^0)=\bar\partial_A(\cO_X)=0$, and $A_1 \simeq \CC\otimes \A$, we have an injection of complexes
$\Omega^\bullet_{\A}\to\mathscr K^\bullet$. We may picture the following diagram.

 \begin{equation*}\label{cplx2}\xymatrix{ & \dots \ar[d] & \dots  \ar[d]  &\dots  \ar[d]  \\
0\ar[r]  & \mathscr K^{k-1} \ar[r]\ar[d]    & \lambda_A^{k-1,0} \ar[r]\ar[d]     & \lambda_A^{k-1,1}  \ar[r]\ar[d]   & \dots \\
0\ar[r]  & \mathscr K^{k} \ar[r]\ar[d]        & \lambda_A^{k,0} \ar[r]\ar[d]        & \lambda_A^{k,1} \ar[r]\ar[d]        & \dots \\
0\ar[r]  & \mathscr K^{k+1} \ar[r]       \ar[d]     & \lambda_A^{k+1,0} \ar[r]         \ar[d]    & \lambda_A^{k+1,0}  \ar[r]\ar[d]   & \dots \\
& \dots & \dots & \dots 
}
\end{equation*}

\begin{defin} We say that $A$ satifies the $\partial_A$-Poincar\'e lemma ($ \bar\partial_A$-Poincar\'e lemma, resp.)  if for every $p$ the sheaf complex $( \lambda_A^{\bullet,p },\partial_A)$
($( \lambda_A^{p,\bullet },\bar\partial_A)$, resp.)
is exact in positive degree. 
\end{defin}

\begin{example} The complex Lie algebroids of Remark \ref{thetanat} 
 satifsfy the $ \bar\partial_A$-Poincar\'e lemma: this is just the exactness
 of the Dolbeault complex twisted by the holomorphic bundle $\A$. Moreover,
 in this case $A_2$ is obviously transitive.
 \end{example}
 
 The following theorem describes the main cohomological features
 of a complex Lie algebrod with a skew-holomorphic structure.

\begin{thm} \label{asimegusta} If $A$ satifies the $ \bar\partial_A$-Poincar\'e lemma, then
\begin{enumerate}
\item (generalized holomorphic de Rham theorem) there is an isomorphism $\mathbb H^p(X,\mathscr K^\bullet)\simeq H^p(A)$, where $\mathbb H^\bullet$ denotes hypercohomology.
\item (generalized Dolbeault theorem)  there are isomorphisms
$$H^p(X,\mathscr K^q) \simeq H^p(\Gamma(\lambda^{q,\bullet}_A),\bar\partial_A)$$
where $H^p(X,\mathscr K^q) $ denotes sheaf cohomology.
\end{enumerate}
If moreover $A$ satifies the $\partial_A$-Poincar\'e lemma, then
\begin{enumerate}  \setcounter{enumi}{2}
\item the sheaf complex $(\Lambda^\bullet A^\ast,d_A)$ is exact in positive degree
(i.e., there is a Poincar\'e lemma for the differential $d_A$);
\item (generalized de Rham theorem) there is an isomorphism $H^p(X,\mathscr F^\infty)\simeq H^p(A)$, where
$H^p(X,\mathscr F^\infty)$ is sheaf cohomology, and 
$\mathscr F^\infty=\ker[\CC_X\stackrel{d_A}{\to}A^\ast]$ is the sheaf of Casimir functions of $A$;
\item the  sheaf complex $(\mathscr K^\bullet,\partial_A)$ is exact in positive degree.
\end{enumerate}
Finally, if  additionally $A_2$ is    transitive,
\begin{enumerate}  \setcounter{enumi}{5}
\item there is an isomorphism $H^p(X,\mathscr F)\simeq H^p(A)$, where $\mathscr F=\ker[\cO_X\stackrel{\partial_A}{\to}\A^\ast]$ is the sheaf of Casimir functions of $\A$;
\item if $X$ is Stein, the cohomology groups $H^p(X,\mathscr F)$ (and therefore the groups $H^p(A)$)
are isomorphic to the cohomology groups of the complex of global sections
of $\Omega_\A^\bullet$.
\end{enumerate}
\end{thm}
\begin{proof} (i) The $E_1$ term of the first spectral sequence of the double complex of sheaves
$(\lambda^{\bullet,\bullet},\partial_A,\bar\partial_A)$ is given by
$$\left\{\begin{array}{ccl} E_1^{p,0} & \simeq &  \mathscr K^p  \\
E_1^{p,q} &=&  0 \quad  \text {for\ } q>0.\end{array}\right.$$
So the spectral sequence degenerates at the second step,
and one has 
\begin{equation*}\label{E2} \left\{\begin{array}{ccl} E_2^{p,0} & \simeq &  \mathscr H^p(\mathscr K^\bullet) \\ E_2^{p,q} &=&  0 \quad  \text {for\ } q>0.\end{array}\right.\end{equation*}
This, together with Proposition \ref{total}, proves that
the composition $\mathscr K^\bullet \hookrightarrow \lambda_A^{\bullet,0} \hookrightarrow
\Lambda^\bullet A^\ast$ is a quasi-isomorphism between 
the complexes $(\Lambda^\bullet A^\ast,d_A)$ and $\mathscr K^\bullet$. Since the sheaves $\Lambda^\bullet A^\ast$ are fine, this yields point (i). 

Point (ii) follows
from the abstract de Rham theorem. 

Point (iii) is obvious. 

This also implies point (iv): $\Lambda^\bullet A^\ast$ is a (fine) resolution of
$\mathscr F^\infty$, so that the abstract de Rham theorem yields the claim.
Point (v) follows from (iii) and the quasi-isomorphism $(\Lambda^\bullet A^\ast,d_A)\simeq\mathscr K^\bullet$.  

(vi) Since $A_2$ is transitive we have $\mathscr K^\bullet\simeq \Omega_\A^\bullet$. 
On the other hand, by (v) the complex formed by $\mathscr F$ in degree zero
is quasi-isomorphic to $\Omega_\A^\bullet$. Therefore,
$$H^p(X,\mathscr F)\simeq\mathbb H^p(X,\Omega^\bullet_\A)\simeq H^p(A)\,.$$

(vii) Since the complex $\Omega_\A^\bullet$ is a resolution of $\mathscr F$,
there is a spectral sequence whose second term is $E_2^{p,q}=H^q(H^p(X,\Omega_\A^\bullet),\partial_A)$, which converges to $H^\bullet(X,\mathscr F)$. If $X$ is Stein  the only nonzero
terms in the second term are $E_2^{0,q}=H^q(\Omega_\A^\bullet(X),\partial_A)$,
whence the claim follows.
\end{proof}

\begin{remark} \label{compa} (i) Note that in points (i) to (v) of Theorem \ref{asimegusta} we do not need to assume that $A_2$ is transitive.

(ii)  If $A_2=T^{0,1}X$ with $a_2=\operatorname{id}$ (see Remark \ref{thetanat}), then
$\mathscr K^\bullet \simeq\Omega_\A^\bullet$, and   point (i) of Theorem \ref{asimegusta}
yields the result in \cite{GSX} (changing   their statement ``the cohomology of the complex
$\Omega_\A^\bullet$'' (which is zero in positive degree) into   ``the hypercohomology of the complex
$\Omega_\A^\bullet$'' ).

(iii) If furthermore $A_1=T^{1,0}X$, with $a_1=\operatorname{id}$, so that
$A$ is the complex de Rham algebroid, point (ii) of Theorem \ref{asimegusta} is Dolbeault theorem,
point (iv) is de Rham theorem,  and
point (i) is the holomorphic de Rham theorem, see \cite{Voisin}.
\end{remark}

\bigskip
\section{Examples}\label{examples}
\subsection{Holomorphic Poisson structures}\label{holoPoisson}
This example has been already considered in \cite{GSX}, however
we briefly describe it here for the sake of completeness. Let $X$ be a complex manifold,
and $P$ a holomorphic Poisson tensor, so that the holomorphic cotangent bundle
$\Omega^1_X$ with the anchor $P\colon \Omega^1_X\to\Theta_X$ is a holomorphic
Lie algebroid. As already discussed, we can match this algebroid with the
Lie algebroid naturally associated to the bundle $T_X^{0,1}$, getting a complex Lie algebroid with skew-holomorphic structure $A$. Theorem \ref{asimegusta}
and Lemma \ref{whentrans} yield the isomorphisms
$$\mathbb H^p(X,\mathscr V_P^\bullet)\simeq H^p(A), \qquad H^p(X,\mathscr V_P^q) \simeq H^p(\Gamma(\lambda^{q,\bullet}_A),\bar\partial_A)$$
where $\mathscr V^\bullet_P$ is the sheaf complex of holomorphic multivector fields
with the differential given by the Poisson tensor $P$.    The first of these isomorphisms
describes the relationship between the coholomogy of a holomorphic Poisson
manifold and the cohomology of the underlying smooth Poisson manifold.

If $X$ is Stein, the hypercohomology 
$\mathbb H^\bullet(X,\mathscr V_P^\bullet)$ is isomorphic to the cohomology of the complex
of global sections of $\mathscr V^\bullet_P$.

\subsection{Holomorphic tangential Lichnerowicz-Poisson cohomology}\label{tangentialPoisson}
Let $P$ be a regular holomorphic Poisson tensor on a complex manifold $X$ (i.e., the rank
of the complex linear map $P_x\colon (\Theta_X^\ast)_x\to(\Theta_X)_x$ does not depend on $x$).
Setting $\A=\Omega^1_X/\ker P$ one gets an exact sequence of holomorphic vector bundles
$$ 0 \to \ker P \to \Omega^1_X \to \A \to 0\,.$$
Moreover the bracket defined by $P$ on the local sections of $\Omega^1_X$
descends to a bracket on $\A$, and one has a morphism $\tilde a\colon \A\to\Theta_X$.
This defines a holomorphic Lie algebroid $\A\stackrel{\tilde a}{\to}\Theta_X$
(this is of course a special case of the situation described in Section \ref{tangential}).
We call the hypercohomology of the associated complex $\Omega_\A^\bullet$
the \emph{holomorphic tangential Lichnerowicz-Poisson cohomology}.
Let $A = [\CC_X\otimes\A]\bowtie T^{0,1}_X$. Then $A$ satisfies
the $\bar\partial_A$-Poincar\'e lemma, and $A_2$ is transitive. Moreover, 
 by Proposition \ref{locex} $A$ satisfies
the $\partial_A$-Poincar\'e lemma, so that $A$ satisfies all properties
required in Theorem \ref{asimegusta}. We have therefore isomorphisms
$$\mathbb H^p(X,\Omega_\A^\bullet)\simeq H^p(X,\mathscr F)\simeq H^p(A)\,.$$
If $X$ is Stein, these groups also coincide with the groups $H^p(\Omega_\A^\bullet(X),\partial_A)$.

\subsection{Skew-holomorphic Poisson structures} Let suppose that
on a complex manifold $X$ we have two holomorphic Poisson tensors $P_1$ and $P_2$. 
The vector bundles $\Omega_X^{1,0}$ and  $\Omega_X^{0,1}$, equipped with the brackets
given by the Poisson tensor $P_1$, and the complex conjugate Poisson tensor $\bar P_2$,
respectively, give rise to complex Lie algebroids $A_1$, $A_2$, with anchors
$P_1\colon A_1\to T^{1,0}_X$, $\bar P_2\colon A_2 \to T^{0,1}_X$. Each algebroid carries
a representation of the other by letting
\begin{equation}\label{cross}\nabla_\alpha\beta = \partial_{P_1(\alpha)}(\beta),\qquad
\nabla_\beta\alpha = \bar\partial_{\bar P_2(\beta)}(\alpha)\end{equation}
if $\alpha\in\Omega^{1,0}(X)$, $\beta\in\Omega^{0,1}(X)$.
\begin{prop} \label{matchingPoisson}
The Lie algebroids $A_1$ and $A_2$, with the module structures given by equation
\eqref{cross}, form a pair of matched Lie algebroids. The matched algebroid
$A=A_1\bowtie A_2$ is a Lie algebroid with skew-holomorphic structure, whose
underlying vector bundle is the  complexified smooth cotangent bundle of $X$. \end{prop}
\begin{proof} We need only to show that the matching conditions are satisfied. As we already noted,
according to Proposition 4.5 of \cite{GSX}, this is tantamount to the commutativity of
the differential $\partial_1$  of the Lie algebroid of $A_1$ twisted by $\Lambda^\bullet A_2^\ast$
with the differential  $\partial_2$ of $A_2$ twisted by $\Lambda^\bullet A_1^\ast$. 
By slightly generalizing the formulas in Proposition 4.25 of \cite{GSX}, we can write
$$ \partial_1(\mu\otimes\nu) = \llbracket P_1,\mu\rrbracket \otimes \nu 
+\sum_i (e_i\wedge\mu)\otimes \Lie_{P_1(e^{\ast i})}\nu$$
where $\mu \in \Gamma(\Lambda^\bullet A_1^\ast) = \Gamma(\Lambda^\bullet T_X^{1,0})$,
$\nu \in \Gamma(\Lambda^\bullet A_2^\ast) = \Gamma(\Lambda^\bullet T_X^{0,1})$,
$\{e_i\}$ is a local basis of sections of $T_X^{1,0}$,  $\{e^{\ast i}\}$  is the
dual basis, $\Lie$ is the Lie derivative, and $\llbracket\, , \, \rrbracket$ is the Schouten bracket. Analogously,
we have 
$$\partial_2(\mu\otimes\nu)  = \mu \otimes \llbracket \bar P_2,\nu\rrbracket
+ \sum_i \Lie_{\bar P_2(f^{\ast i})} \mu \otimes (f_i\wedge \nu)$$
where $\{f_i\}$ is a local basis of sections of $T_X^{0,1}$,  and $\{f^{\ast i}\}$
is the dual basis. Since both differentials obey a Leibniz rule, it is enough to verify their commutativity when $\mu$ is holomorphic, and $\nu$ is antiholomorphic. We thus obtain
$$\partial_1(\mu\otimes\nu) = \llbracket P_1,\mu \rrbracket \otimes\nu= d_1 \mu \otimes \nu $$
where $d_1$ is the differential of the (untwisted) Lichnerowicz-Poisson complex of $P_1$.
Analogously, 
$$\partial_2(\mu\otimes\nu) = \mu \otimes \llbracket \bar P_2,\nu\rrbracket = \mu \otimes d_2 \nu$$
where $d_2$ is the differential of the  Lichnerowicz-Poisson complex of $\bar P_2$.
Since $d_1\mu$ is again holomorphic, and $d_2\nu$ is again antiholomorphic, we have
$$ \partial_1\partial_2 (\mu\otimes\nu)  = d_1\mu\otimes d_2\nu = \partial_2\partial_1(\mu\otimes\nu).$$ 
\end{proof}
This results strengthens the remark already done in \cite{Cordero00}, where it is noted
that the Schouten bracket of $P_1$ and $\bar P_2$ vanishes, i.e., $P_1$ and $\bar P_2$ 
satisfy a bihamiltonian condition. However, as we discuss in \cite{UgoVolyBiham},
the matching pair condition is stronger than the bihamiltonian condition,
and the former indeed implies the latter.

In this generality, the skew-holomorphic Lie algebroid $A$ satisfies none of the  conditions
of Theorem \ref{asimegusta}.
If the Poisson tensor $P_2$ is nondegenerate, it establishes
a Lie algebroid isomorphism $\bar P_2\colon A_2 \to T^{0,1}_X$, where
$T^{0,1}_X$ is given its standard Lie algebroid structure. Thus we recover the example of
Section \ref{holoPoisson}.

\subsection{Skew-holomorphic tangential Poisson structures} We can mix the two previous
examples considering two regular holomorphic Poisson tensors $P_1$, $P_2$ and matching
the Lie algebroid obtained as in Section \ref{tangentialPoisson} from $P_1$ with the
complex conjugate of the one obtained from $P_2$. In this way we obtain a situation
where the results (i) to (v) of Theorem \ref{asimegusta} hold.

\section{Conclusions and Perspectives}

The theory we have developed in this paper provides  generalizations and simplifications of known results and their proofs (\cite{GSX}). Moreover, quite natural applications of the cohomology theory we develop in this paper may be found in connection with pairs of holomorphic Poisson structures
that satisfy some compatibility condition, as it happens for holomorphic bihamiltonian system \cite{UgoVolyBiham}. One can therefore foresee applications to holomorphic integrable systems. More generally, one could envisage applications to the study of complex manifolds equipped with a pair of compatible, possibly singular, holomorphic foliations.

This cohomology theory can be used also to study a  deformation theory for complex Lie algebroids,
and a relation of deformation conditions with the above-mentioned compatibility conditions.
Recently an example of this connection was proposed in the paper \cite{CarJoanSantos}, motivated by
some natural questions of Classical Mechanics. Our constructions, and their natural generalizations
to the ``matched'' algebroids of \cite{CarJoanSantos}, their dual algebroids and symplectic realizations
of the matched co-algebroids, may provide a useful toolbox and a transparent language to study the
``internal'' deformations and the compatibility conditions in the holomorphic case.

\end{document}